\documentclass[11pt]{article}

\usepackage[T1]{fontenc}
\usepackage[utf8]{inputenc}
\usepackage{lmodern}
\usepackage{amsmath,amssymb,amsthm,mathtools,bm}
\usepackage[margin=1in]{geometry}
\usepackage{microtype}
\usepackage{booktabs,tabularx,array}
\usepackage{enumitem}
\usepackage{algorithm,algpseudocode}
\usepackage{float}
\usepackage{needspace,etoolbox}
\usepackage[hidelinks]{hyperref}
\usepackage{bookmark}
\usepackage[nameinlink,capitalize,noabbrev]{cleveref}

\newcommand{\ResearchAgentSystem}{our laboratory's internal auto-research system}

\newif\ifanonymous
\anonymousfalse

\newcommand{\AuthorOne}{Haihan Zhang}
\newcommand{\AuthorTwo}{Wendao Wu}
\newcommand{\AuthorThree}{Chenheng Zhang}
\newcommand{\AuthorFour}{Haoxuan Li}
\newcommand{\AuthorFive}{Zhouchen Lin}
\newcommand{\AuthorSix}{Cong Fang}
\newcommand{\AuthorSeven}{Yanyi Li}
\newcommand{\AuthorEight}{Chunyuan Zheng}
\newcommand{\AffiliationOne}{Peking University}

\newcommand{\EmailOne}{colombo.2026@outlook.com}

\ifanonymous
  \newcommand{\PDFAuthors}{Anonymous Authors}
\else
  \newcommand{\PDFAuthors}{\AuthorOne; \AuthorTwo; \AuthorThree; \AuthorFour; \AuthorFive}
\fi

\hypersetup{
  pdftitle={Near-Optimal Deterministic Exact-Value Complexity for Smooth Convex Optimization},
  pdfauthor={\PDFAuthors},
  pdfsubject={Exact function values, smooth convex optimization, adaptive deterministic algorithms, fixed-objective lower bounds, and resisting rotations}
}

\allowdisplaybreaks[1]

\setlist{leftmargin=*,itemsep=2pt,topsep=3pt}
\numberwithin{equation}{section}

\newcommand{\R}{\mathbb R}
\newcommand{\E}{\mathbb E}
\newcommand{\Prob}{\mathbb P}

\newcommand{\ip}[2]{\langle #1,#2\rangle}
\newcommand{\norm}[1]{\lVert #1\rVert}
\newcommand{\dist}{\operatorname{dist}}
\newcommand{\Lip}{\operatorname{Lip}}

\newcommand{\argminop}{\operatorname*{arg\,min}}

\newcommand{\Span}{\operatorname{span}}
\newcommand{\cF}{\mathcal F}

\newcommand{\cC}{\mathcal C}

\newcommand{\1}{\bm{1}}

\newcolumntype{P}[1]{>{\raggedright\arraybackslash}p{#1}}
\newcolumntype{Y}{>{\raggedright\arraybackslash}X}

\newtheorem{theorem}{Theorem}[section]
\newtheorem{lemma}[theorem]{Lemma}
\newtheorem{proposition}[theorem]{Proposition}
\newtheorem{corollary}[theorem]{Corollary}

\theoremstyle{definition}
\newtheorem{definition}[theorem]{Definition}

\theoremstyle{remark}
\newtheorem{remark}[theorem]{Remark}

\AtBeginEnvironment{theorem}{\Needspace{5\baselineskip}}
\AtBeginEnvironment{lemma}{\Needspace{5\baselineskip}}
\AtBeginEnvironment{proposition}{\Needspace{5\baselineskip}}
\AtBeginEnvironment{corollary}{\Needspace{5\baselineskip}}
\AtBeginEnvironment{definition}{\Needspace{5\baselineskip}}
\AtBeginEnvironment{assumption}{\Needspace{5\baselineskip}}
\AtBeginEnvironment{proof}{\Needspace{3\baselineskip}}

\crefname{theorem}{Theorem}{Theorems}
\crefname{lemma}{Lemma}{Lemmas}
\crefname{proposition}{Proposition}{Propositions}
\crefname{corollary}{Corollary}{Corollaries}
\crefname{section}{Section}{Sections}
\crefname{equation}{Eq.}{Eqs.}

\title{\textbf{Near-Optimal Deterministic Exact-Value Complexity for
Smooth Convex Optimization}}

\ifanonymous
  \author{Anonymous Authors}
\else
  \author{
    \AuthorTwo$^{1,*}$ \quad
    \AuthorOne$^{1,*}$ \quad
    \AuthorThree$^{1,*}$\\[0.20em]
    \AuthorSeven$^{1}$ \quad
    \AuthorEight$^{1}$\\[0.20em]
    \AuthorSix$^{1,\dagger}$ \quad
    \AuthorFour$^{1,\dagger}$ \quad
    \AuthorFive$^{1,\dagger}$\\[0.60em]
    \small $^{1}$\AffiliationOne\\[0.35em]
    \small
    \href{mailto:\EmailOne}{\texttt{\EmailOne}} \\[0.35em]
    \small $^{*}$Equal contribution.    \small $^{\dagger}$Corresponding authors.
  }
\fi
\date{}

\begin{document}
\maketitle

\begin{abstract}
We study the deterministic oracle complexity of minimizing a globally
$\beta$-smooth convex function when the only oracle feedback is one exact
real function value.  The algorithm may perform arbitrary exact-real
computation, but every query and the terminal output must lie in the public
ball $B_2^d(R)$, and the unique minimizer is promised to lie in
$B_2^d(R/2)$.  Writing $Q=\beta R^2/\epsilon$, we first give an explicit
bounded-query upper bound
\[
  N_{\epsilon,\mathrm{det-ad}}^{\mathrm{val}}(d,R,\beta)
  \le 6d\sqrt Q
  \qquad(0<\epsilon\le\beta R^2),
\]
obtained by coordinate finite differences coupled with an error-robust
accelerated projected method.  The main contribution is a lower
bound that matches this upper bound whenever the square-root branch is active: there are universal constants
$c,c_\epsilon>0$ such that, for all sufficiently large $d$ and
$0<\epsilon\le c_\epsilon\beta R^2$,
\[
  N_{\epsilon,\mathrm{det-ad}}^{\mathrm{val}}(d,R,\beta)
  \ge c d\min\left\{
    \sqrt Q,
    \left(\frac{d}{\log(ed)}\right)^{1/3}
  \right\}.
\]
Consequently the minimax complexity is
$\Theta(d\sqrt{\beta R^2/\epsilon})$ throughout the moderate-accuracy range
\[
  \beta R^2\left(\frac{\log(ed)}{d}\right)^{2/3}
  \le \epsilon \le c_\epsilon\beta R^2.
\]
The lower bound is not a noisy-information or independent-coordinate
argument: exact scalar values can encode arbitrarily much information, so
we build one fixed smooth convex objective whose transcript is consistent
with every adaptive reply.  The construction combines a Moreau-smoothed
biased max chain, an exact prefix-shielding identity, and a batched
delayed-rotation compiler.  The present theorem settles the deterministic
bounded-query square-root branch, while leaving the high-accuracy regime,
randomized algorithms, and unrestricted query locations as separate open
extensions.
\end{abstract}


\noindent\textbf{AI Usage.}
Nearly the entire research pipeline for this paper was carried out by, Colombo, 
\ResearchAgentSystem{}, powered by GPT-5.6 Sol. The system also conducted a
Lean-backed article audit of the resulting manuscript. The authors subsequently reviewed and approved the
mathematical claims, presentation, and formal artifacts, and take
responsibility for the final manuscript.The complete Lean audit report and the system's technical report will be made public at a later date.

\paragraph{Publication Intent and Recognition of Independent Work.} 
This preprint is released as a research output of our study of
automated scientific research, for public scrutiny and discussion.
As we do not intend to submit this manuscript to a journal or
conference, it has not undergone the full editorial refinement
and internal review we would normally require for submission.
The authors remain responsible for its contents and for correcting
any errors or inaccurate attribution.

We recognize that releasing these results may affect researchers
who have independently been pursuing the same questions, including
work that is not yet public.
If this applies to your work, we invite you to contact us at
\texttt{colombo.2026@outlook.com}.
A brief description of the overlap, together with any dated
materials you are comfortable sharing, would help us understand
its scope and chronology.
Such materials might include draft excerpts, presentation slides,
or revision records; a complete manuscript or unpublished proof
is not required to begin the conversation.

We will review these cases carefully and update citations,
acknowledgments, and accounts of independent or concurrent work
as appropriate, including revising priority claims where warranted.
Nonpublic materials will be treated confidentially and used only
to assess the overlap unless their authors authorize otherwise.
We will obtain explicit permission from the authors before
publicly discussing such materials or providing them to an
automated research system.

\section{Introduction}
\label{sec:introduction}
A zeroth-order oracle returns only the scalar value of an objective at the
queried point.  In stochastic or finite-precision models, the weakness of
such feedback is visible through standard information inequalities.  The
exact-value model is more subtle.  A single real number may encode
arbitrarily much information, and a fully adaptive deterministic algorithm
may choose each new query from the complete exact transcript.  As a result,
lower bounds based on finite alphabets, independent noisy observations, or
coordinate-wise information accounting do not directly apply.

This paper studies the exact-value analogue of smooth convex minimization.
The unknown objective $f:\R^d\to\R$ is convex, continuously differentiable,
and has globally $\beta$-Lipschitz gradient.  The algorithm is allowed
unlimited exact-real internal computation, but all oracle queries and the
terminal output must lie in the public ball
\[
  X=B_2^d(R).
\]
The minimizer is promised to be unique and to lie in the interior ball
$B_2^d(R/2)$.  The performance criterion is objective suboptimality
$f(\widehat x)-f^\star\le\epsilon$, and the natural scale is
\[
  Q=\frac{\beta R^2}{\epsilon}.
\]

\paragraph{The unresolved exact-value smooth frontier.}
The upper-bound side is immediate in spirit but still needs careful
bounded-query accounting: use exact values to form finite-difference
gradients and then run an accelerated first-order method.  This gives the
classical $d\sqrt Q$ rate.  The lower-bound side is the real obstruction.
Unlike noisy zeroth-order optimization, exact values do not limit the number
of transmitted bits.  Unlike first-order lower bounds, a value-only
algorithm is not forced to reveal a gradient direction.  Unlike nonsmooth
max-affine hard instances, a smooth hard instance must preserve global
Lipschitz continuity of the gradient.  Finally, a resisting-oracle proof is
valid only if all adaptive replies are values of one fixed objective chosen
for the algorithm, not values of objectives that change from round to round.

Recent exact-value work has clarified the nonsmooth convex setting, where
near-quadratic dimension lower bounds show that scalar exact values can be
surprisingly weak in high dimension.  The smooth setting has a different
shape: finite differences provide an $O(d\sqrt Q)$ upper bound, but before
the present argument it was not clear whether a deterministic adaptive
algorithm could exploit exact real values to beat this rate in the
moderate-accuracy regime, or whether the fixed-objective consistency
barrier could be overcome by a smooth lower-bound construction.

\paragraph{Our answer.}
We prove that the square-root branch is optimal under bounded deterministic
queries.  For all sufficiently large dimensions,
\begin{equation}
\label{eq:intro-main-lower}
  N_{\epsilon,\mathrm{det-ad}}^{\mathrm{val}}(d,R,\beta)
  \ge
  c d\min\left\{
    \sqrt{\frac{\beta R^2}{\epsilon}},
    \left(\frac{d}{\log(ed)}\right)^{1/3}
  \right\},
\end{equation}
while the finite-difference accelerated method gives
\begin{equation}
\label{eq:intro-main-upper}
  N_{\epsilon,\mathrm{det-ad}}^{\mathrm{val}}(d,R,\beta)
  \le
  6d\sqrt{\frac{\beta R^2}{\epsilon}}.
\end{equation}
Combining the two bounds yields
\begin{equation}
\label{eq:intro-matching-rate}
  \boxed{
  N_{\epsilon,\mathrm{det-ad}}^{\mathrm{val}}(d,R,\beta)
  =\Theta\left(d\sqrt{\frac{\beta R^2}{\epsilon}}\right)}
\end{equation}
whenever
\[
  \beta R^2\left(\frac{\log(ed)}{d}\right)^{2/3}
  \le \epsilon \le c_\epsilon\beta R^2.
\]
Thus, in the regime where the finite-difference acceleration branch has
not yet reached the high-dimensional localization barrier, exact scalar
values are no more powerful than using them to reconstruct gradients at a
coordinate cost $d$.

The lower bound also identifies the present limitation of the technique.
Its hard chain length $m$ must obey
$m^3\log(ed)\lesssim d$, so the lower bound saturates at
$d^{4/3}/\log^{1/3}(ed)$ at very high accuracy.  This is why the result is
near-optimal rather than a complete characterization of all accuracy scales.
The high-accuracy endpoint, randomized policies, and unbounded query
locations remain outside the theorem's scope.

\subsection{Contributions}
The paper makes five main contributions.
\begin{enumerate}[label=\textup{(C\arabic*)}]
\item \textbf{An explicit deterministic upper bound with bounded queries.}
We give an explicit finite-difference implementation of an accelerated
projected method and prove
\[
  N_{\epsilon,\mathrm{det-ad}}^{\mathrm{val}}(d,R,\beta)
  \le(d+1)\left\lceil2\sqrt{\beta R^2/\epsilon}\right\rceil
  \le6d\sqrt{\beta R^2/\epsilon}.
\]
The proof tracks deterministic gradient error, projection geometry, and all
finite-difference query locations.

\item \textbf{An exact shielding identity for a smooth value chain.}
We smooth a biased max chain by a Moreau envelope and use its simplex dual
to prove a value-level prefix-shielding identity.  Once a tail of hidden
coordinates lies inside a small window, every dual optimizer is supported
on the visible prefix, so the exact value is independent of all later
coordinates.

\item \textbf{A batched delayed-rotation compiler.}
Adaptive queries are grouped into blocks of order $d$.  During a block, the
next hidden direction has not yet been selected; after the block is fixed,
that direction is chosen orthogonal to the whole block and nearly
orthogonal to the past.  A sentinel direction additionally shields the
terminal output.

\item \textbf{A fixed-objective deterministic lower bound.}
The online transcript is realized by one globally smooth convex objective
$F_U$, with a unique minimizer in $B_2^d(R/2)$ and strong-convexity modulus
$\beta/(4096m^2)$.  This proves that fewer than $\Theta(dm)$ exact-value
queries leave error $\Omega(\beta R^2/m^2)$ whenever
$m^3\log(ed)\lesssim d$.

\item \textbf{A clean rate frontier and limitation statement.}
Optimizing the integer lower bound over $m$ gives
\eqref{eq:intro-main-lower}, which matches the upper bound throughout the
moderate-accuracy range.  The same calculation exposes the cubic saturation
that prevents this construction from closing the high-accuracy endpoint.
\end{enumerate}

\subsection{Technical overview}
The upper bound converts exact function values into controlled first-order
information.  At a base point $y\in B_2^d(R/2)$, the $d+1$ values
$f(y),f(y+he_1),\ldots,f(y+he_d)$ produce a forward-difference gradient
$g_h(y)$ with error at most $\beta h\sqrt d/2$.  Choosing $h$ sufficiently
small makes this an inexact first-order oracle on the inner ball.  A
standard estimate-sequence proof, written with explicit lower and upper
oracle errors, gives an accelerated suboptimality bound.  Because the base
points stay in $B_2^d(R/2)$ and $h\le R/2$, all displaced queries remain
inside $B_2^d(R)$.

The lower bound starts from the biased chain
\[
  g_{m,\rho}(y)=\max_{1\le i\le m}\{y_i-8\rho i\},
\]
and replaces it by its Moreau envelope $h_{m,\rho}$.  The dual form is a
maximization over the simplex.  If all coordinates from index $j$ onward
are in $[-\rho,\rho]$, moving any mass placed after $j$ to coordinate $j$
strictly improves the dual objective.  Hence every optimizer is supported
on the prefix $\{1,\ldots,j\}$, and the smoothed exact value does not see
the tail.

The chain is embedded through an orthonormal frame
$U=(u_1,\ldots,u_m)$.  Queries are processed in blocks of size
$b=\lfloor d/4\rfloor$.  The replies in block $j$ depend only on
$u_1,\ldots,u_{j-1}$, so the full adaptive block is known before $u_j$ is
chosen.  A spherical-cap argument then selects $u_j$ orthogonal to that
block and with small projection on earlier queries.  After the final query,
a sentinel direction is chosen orthogonal to the output.  Exact shielding
then proves that the transcript coincides pointwise with the final fixed
objective.  Since the output has not progressed to the last chain
coordinate, comparison with a point moving along all hidden directions gives
an error of order $\beta R^2/m^2$.

The geometric bottleneck is the simultaneous small-projection condition.
With $\rho\asymp R/m^{3/2}$, cap avoidance requires
$m^3\log(ed)\lesssim d$.  Therefore the lower bound matches
$d\sqrt Q$ while $\sqrt Q\lesssim(d/\log(ed))^{1/3}$ and saturates beyond
that threshold.

\subsection{Scope and organization}
The theorem is an information-complexity result for deterministic adaptive
algorithms, exact scalar function values, and bounded query/output
locations.  It does not claim polynomial arithmetic complexity, randomized
lower bounds, or unrestricted-query lower bounds.  The hard objectives are
$C^1$ with globally Lipschitz gradients, but the construction should not be
quoted as a $C^2$ or $C^\infty$ lower bound without an additional smoothing
step.

\cref{sec:related} compares the result with nearby oracle models and rate
statements.  \cref{sec:main} states the formal model and main theorems.
\cref{sec:upper} proves the upper bound.  \cref{sec:overview}--\cref{sec:gap}
construct the smooth fixed-objective lower bound, and \cref{sec:rate}
optimizes the integer theorem into the dimension--accuracy statement.
\cref{sec:discussion} records the precise near-optimal interpretation and
limitations.

\section{Related work and rate comparison}
\label{sec:related}
\paragraph{Exact values versus noisy values.}
Stochastic and randomized zeroth-order optimization has a mature theory in
which one or two noisy function evaluations are converted into randomized
gradient estimates.  Representative results include the lower bounds of
Shamir~\cite{shamir2013complexity}, the two-point rates of Duchi et
al.~\cite{duchi2015optimal}, and random gradient-free methods such as
Nesterov and Spokoiny~\cite{nesterov2017random}.  Those results are not
pathwise exact-real lower bounds: their information restrictions come from
noise, randomness, or expected performance criteria.

\paragraph{Smooth first-order baselines.}
For smooth convex minimization with exact gradients, accelerated methods
achieve the familiar $O(\sqrt Q)$ iteration complexity, with matching
information-based lower bounds in standard first-order models; see, for
example, Nemirovski and Yudin~\cite{nemirovski1983problem} and Drori and Teboulle~\cite{drori2017exact}.  A value-only algorithm
can simulate one gradient by $d+1$ coordinate finite differences, leading
to the $O(d\sqrt Q)$ upper bound proved here with bounded-query details.
The contribution of this paper is the corresponding deterministic
exact-value lower bound in the moderate-accuracy regime.

\paragraph{Exact-value convex lower bounds.}
Recent work on nonsmooth Lipschitz convex optimization shows that exact
scalar values do not automatically make derivative-free optimization easy:
near-quadratic dimension lower bounds have been proved in high-dimensional
settings~\cite{kerger2026closing,zhang2026nearoptimal}.  Smoothness changes
the hard-instance geometry.  A max-affine construction is no longer
admissible, and a generic smoothing may leak information about many hidden
pieces through one exact value.  The Moreau-envelope shielding identity
below is designed to retain both smoothness and exact transcript
invariance.

\paragraph{Resisting rotations and fixed transcripts.}
Delayed rotations and zero-chain ideas are common in oracle lower bounds
for convex and nonconvex optimization, including parallel and randomized
settings~\cite{carmon2020lowerI,carmon2020lowerII,diakonikolas2019parallel}.  The exact-value
model imposes an additional consistency requirement: the answers returned
online must be values of one final fixed objective.  Our batched compiler
selects directions after adaptive blocks are fixed, but the exact shielding
lemma proves that no previous reply changes when the frame is completed.

To state the rate comparison, write
\[
  Q=\frac{\beta R^2}{\epsilon}.
\]
The table suppresses universal constants and logarithmic factors where
indicated; in the last row, write $\bar d=d/\log(ed)$.

\begin{table}[H]
\centering
\caption{Representative rates for value and first-order convex optimization models.  Here $Q=\beta R^2/\epsilon$ and $\bar d=d/\log(ed)$.  Rows with different feedback and success conventions are not directly comparable.}
\label{tab:related-rates}
\scriptsize
\setlength{\tabcolsep}{2.0pt}
\renewcommand{\arraystretch}{1.18}
\begin{tabularx}{\linewidth}{@{}P{1.95cm}P{1.90cm}P{2.70cm}Y P{1.65cm}@{}}
\toprule
Setting & Feedback model & Representative complexity & Scope / caveat & Reference \\
\midrule
Smooth convex minimization
& exact gradients
& $\Theta(\sqrt Q)$
& Classical first-order oracle model; not value-only.
& \cite{drori2017exact} \\
\addlinespace
Smooth convex minimization
& exact scalar values
& $O(d\sqrt Q)$
& Coordinate finite differences plus accelerated first-order method; this paper gives bounded-query accounting.
& This work, upper bound \\
\addlinespace
Nonsmooth Lipschitz convex minimization
& exact scalar values
& near $d^2$ lower bounds in high dimension
& Different regularity class; hard instances are nonsmooth support or max-type functions.
& \cite{kerger2026closing,zhang2026nearoptimal} \\
\addlinespace
Continuous / general convex localization
& exact scalar values
& $\widetilde O(d^2)$
& Protasov's bound is $O(d^2\log(d+1)\log(1/\delta))$ for relative error; not a smooth-specific lower bound.
& \cite{protasov1996algorithms} \\
\addlinespace
Smooth convex minimization
& exact scalar values
& $\Omega(d\min\{\sqrt Q,\bar d^{1/3}\})$
& Deterministic adaptive algorithms; bounded queries and outputs; one fixed globally smooth convex objective.
& This work, lower bound \\
\bottomrule
\end{tabularx}
\end{table}

The last row matches the finite-difference upper bound exactly when
$Q\le(d/\log(ed))^{2/3}$.  Beyond that threshold the current lower bound
saturates, so the paper deliberately separates the matched square-root
branch from the still-open high-accuracy endpoint.

\section{Problem formulation and main result}\label{sec:main}
\label{sec:model}

\subsection{Objective class and exact-value algorithms}\label{sec:model-class}

Throughout, $\log$ denotes the natural logarithm, $d\in\mathbb{N}$, and $R,\beta>0$.  Let
\[
X:=B_2^d(R)=\{x\in\R^d:\norm{x}_2\le R\}.
\]

\begin{definition}[Objective class]
\label{def:class}
Let $\cF_\beta(d,R)$ be the set of all functions $f:\R^d\to\R$ such that
\begin{enumerate}[label=(\roman*)]
\item $f$ is convex and continuously differentiable;
\item $\nabla f$ is globally $\beta$-Lipschitz in Euclidean norm;
\item $f$ has a unique global minimizer $x_f^\star\in B_2^d(R/2)$.
\end{enumerate}
Write $f^\star=\min_x f(x)$.
\end{definition}

A deterministic adaptive algorithm with budget $T$ consists of Borel maps
\[
Q_t:(X\times\R)^{t-1}\to X,
\qquad t=1,\ldots,T,
\]
and a Borel terminal map
\[
Q_{\mathrm{out}}:(X\times\R)^T\to X.
\]
Dependence on the public parameters $(d,R,\beta,\epsilon)$ is suppressed.  Recursively,
\[
x_t=Q_t((x_1,f(x_1)),\ldots,(x_{t-1},f(x_{t-1}))),
\]
and
\[
\widehat x=Q_{\mathrm{out}}((x_1,f(x_1)),\ldots,(x_T,f(x_T))).
\]
Repeated calls are charged unless a stored value is reused.  Internal exact-real computation is free, and the objective is fixed throughout the interaction.

\begin{definition}[Minimax exact-value complexity]
For $\epsilon>0$, define
\[
N_{\epsilon,\mathrm{det-ad}}^{\mathrm{val}}(d,R,\beta)
:=
\inf\left\{
T\in\mathbb{N}_0:
\begin{array}{l}
\text{there exists a deterministic adaptive $T$-call algorithm}\\[1mm]
\text{such that }\sup_{f\in\cF_\beta(d,R)}
[f(\widehat x_f)-f^\star]\le\epsilon
\end{array}
\right\}.
\]
As usual, $\inf\varnothing=+\infty$.
\end{definition}

\subsection{Main upper and lower bounds}\label{sec:main-results}

\begin{theorem}[Zeroth-order accelerated upper bound]
\label{thm:upper}
For every $d\in\mathbb{N}$, $R,\beta>0$, and $0<\epsilon\le\beta R^2$,
\begin{equation}
\label{eq:upper-rate}
N_{\epsilon,\mathrm{det-ad}}^{\mathrm{val}}(d,R,\beta)
\le
(d+1)\left\lceil2\sqrt{\frac{\beta R^2}{\epsilon}}\right\rceil
\le
6d\sqrt{\frac{\beta R^2}{\epsilon}}.
\end{equation}
All oracle calls made by the algorithm lie in $B_2^d(R)$, and its output lies in $B_2^d(R/2)$.
\end{theorem}

The lower bound is most transparent in an integer-parameter form.

\begin{theorem}[Integer-parameter lower bound]
\label{thm:integer}
There are universal constants $c_0>0$ and $d_0\in\mathbb{N}$ such that the following holds.  Let $d\ge d_0$ and let $m$ be an integer satisfying
\begin{equation}
\label{eq:m-conditions}
1\le m\le\frac d4,
\qquad
m^3\log(ed)\le c_0d.
\end{equation}
Every deterministic adaptive algorithm making fewer than $dm/8$ exact-value calls admits a hard function $F\in\cF_\beta(d,R)$ for which
\begin{equation}
\label{eq:integer-gap}
F(\widehat x)-F^\star
\ge
2^{-17}\frac{\beta R^2}{m^2}.
\end{equation}
Moreover, the constructed hard function is $\beta/(4096m^2)$-strongly convex; in fact, its global strong-convexity modulus equals this value.
\end{theorem}

\begin{theorem}[Dimension--accuracy lower bound]
\label{thm:main}
There exist universal constants $c,c_\epsilon>0$ and $d_0\in\mathbb{N}$ such that for every
\[
d\ge d_0,
\qquad R,\beta>0,
\qquad 0<\epsilon\le c_\epsilon\beta R^2,
\]
one has
\begin{equation}
\label{eq:main-rate}
\boxed{
N_{\epsilon,\mathrm{det-ad}}^{\mathrm{val}}(d,R,\beta)
\ge
c d\min\left\{
\sqrt{\frac{\beta R^2}{\epsilon}},
\left(\frac{d}{\log(ed)}\right)^{1/3}
\right\}.}
\end{equation}
\end{theorem}

\begin{corollary}[Matching moderate-accuracy complexity]
\label{cor:matching}
There exist universal constants $c,C,c_\epsilon>0$ and $d_0\in\mathbb{N}$ such that whenever
\begin{equation}
\label{eq:matching-regime}
d\ge d_0,
\qquad
\beta R^2\left(\frac{\log(ed)}{d}\right)^{2/3}
\le\epsilon\le c_\epsilon\beta R^2,
\end{equation}
one has
\begin{equation}
\label{eq:matching-rate}
c d\sqrt{\frac{\beta R^2}{\epsilon}}
\le
N_{\epsilon,\mathrm{det-ad}}^{\mathrm{val}}(d,R,\beta)
\le
C d\sqrt{\frac{\beta R^2}{\epsilon}}.
\end{equation}
Thus the minimax rate is determined up to universal constants throughout \eqref{eq:matching-regime}.
\end{corollary}

\begin{corollary}[Two lower-bound regimes]
\label{cor:regimes}
Put $Q=\beta R^2/\epsilon$.  Subject to the universal thresholds in \cref{thm:main}:
\begin{enumerate}[label=(\roman*)]
\item if $Q\le(d/\log(ed))^{2/3}$, then
\[
N_{\epsilon,\mathrm{det-ad}}^{\mathrm{val}}(d,R,\beta)
\ge c d\sqrt Q;
\]
\item if $Q\ge(d/\log(ed))^{2/3}$, then
\[
N_{\epsilon,\mathrm{det-ad}}^{\mathrm{val}}(d,R,\beta)
\ge c\frac{d^{4/3}}{\log^{1/3}(ed)}.
\]
\end{enumerate}
\end{corollary}

\section{Complete proof of the upper bound}
\label{sec:upper}

The upper bound optimizes over the inner ball
\[
\cC:=B_2^d(R/2),
\]
which contains the promised minimizer.  The diameter of $\cC$ is $R$.  Every finite-difference base point lies in $\cC$, so a sufficiently short coordinate displacement remains inside the allowed query ball $X=B_2^d(R)$.

\subsection{A deterministic finite-difference oracle}

\begin{lemma}[Forward differences]
\label{lem:fd}
Let $f\in\cF_\beta(d,R)$, $y\in\cC$, and $0<h\le R/2$.  Define
\begin{equation}
\label{eq:fd}
[g_h(y)]_i
:=
\frac{f(y+he_i)-f(y)}{h},
\qquad i=1,\ldots,d.
\end{equation}
Then all $d+1$ queried points belong to $X$, and
\begin{equation}
\label{eq:fd-error}
\norm{g_h(y)-\nabla f(y)}_2
\le\frac{\beta h\sqrt d}{2}.
\end{equation}
\end{lemma}

\begin{proof}
The triangle inequality gives
\[
\norm{y+he_i}_2\le\norm y_2+h\le R,
\]
so every query is admissible.  Smoothness along the line $t\mapsto y+te_i$ gives
\[
\left|f(y+he_i)-f(y)-h\,\partial_i f(y)\right|
\le\frac{\beta h^2}{2}.
\]
After division by $h$, each coordinate error is at most $\beta h/2$.  Summing the squared coordinate bounds proves \eqref{eq:fd-error}.
\end{proof}

The next lemma turns a norm-accurate gradient into a two-sided inexact first-order model on the bounded set $\cC$.

\begin{lemma}[Inexact oracle inequalities]
\label{lem:inexact-oracle}
Let $x,y\in\cC$ and suppose
\[
\norm{g-\nabla f(y)}_2\le\delta_g.
\]
Set
\begin{equation}
\label{eq:oracle-errors}
\overline\beta:=2\beta,
\qquad
\delta_\ell:=R\delta_g,
\qquad
\delta_u:=\frac{\delta_g^2}{2\beta}.
\end{equation}
Then
\begin{align}
f(x)
&\ge f(y)+\ip{g}{x-y}-\delta_\ell,
\label{eq:oracle-lower}\\
f(x)
&\le f(y)+\ip{g}{x-y}
+\frac{\overline\beta}{2}\norm{x-y}_2^2+\delta_u.
\label{eq:oracle-upper}
\end{align}
\end{lemma}

\begin{proof}
Write $e=g-\nabla f(y)$.  Convexity gives
\[
f(x)\ge f(y)+\ip{\nabla f(y)}{x-y}
=f(y)+\ip g{x-y}-\ip e{x-y}.
\]
Since $\operatorname{diam}(\cC)=R$, the last inner product is at most $R\delta_g$ in absolute value, proving \eqref{eq:oracle-lower}.

By $\beta$-smoothness,
\begin{align*}
f(x)
&\le f(y)+\ip{\nabla f(y)}{x-y}
+\frac\beta2\norm{x-y}_2^2\\
&\le f(y)+\ip g{x-y}
+\delta_g\norm{x-y}_2
+\frac\beta2\norm{x-y}_2^2.
\end{align*}
Young's inequality
\[
\delta_g r\le\frac\beta2r^2+\frac{\delta_g^2}{2\beta}
\]
then gives \eqref{eq:oracle-upper} with $\overline\beta=2\beta$.
\end{proof}

\subsection{An error-robust accelerated projected method}

Fix an initial point $x_0=z_0=0\in\cC$, let $A_0=0$, and define
\[
\psi_0(x):=\frac12\norm{x-x_0}_2^2,
\qquad x\in\cC.
\]
For $k=0,\ldots,N-1$, choose the unique $\alpha_{k+1}>0$ satisfying
\begin{equation}
\label{eq:alpha}
A_{k+1}:=A_k+\alpha_{k+1}
=\overline\beta\alpha_{k+1}^2.
\end{equation}
Given $x_k,z_k\in\cC$, set
\begin{equation}
\label{eq:upper-y}
y_{k+1}
:=
\frac{A_kx_k+\alpha_{k+1}z_k}{A_{k+1}}.
\end{equation}
At $y_{k+1}$ obtain $f(y_{k+1})$ and an approximate gradient $g_{k+1}$ satisfying the oracle inequalities in \cref{lem:inexact-oracle}.  Update
\begin{align}
\psi_{k+1}(x)
&:=
\psi_k(x)+\alpha_{k+1}
\left[f(y_{k+1})+\ip{g_{k+1}}{x-y_{k+1}}\right],
\label{eq:psi-update}\\
z_{k+1}
&:=\argminop_{x\in\cC}\psi_{k+1}(x),
\label{eq:z-update}\\
x_{k+1}
&:=
\frac{A_kx_k+\alpha_{k+1}z_{k+1}}{A_{k+1}}.
\label{eq:x-update}
\end{align}
All three points $y_{k+1},z_{k+1},x_{k+1}$ belong to $\cC$.  The subproblem \eqref{eq:z-update} is computationally explicit:
\[
z_{k+1}
=
\Pi_{\cC}\left(x_0-\sum_{i=1}^{k+1}\alpha_i g_i\right),
\]
because the remaining terms in $\psi_{k+1}$ are constant in $x$.

\begin{lemma}[Accelerated estimate with deterministic oracle errors]
\label{lem:accelerated-error}
Suppose the same nonnegative numbers $\delta_\ell,\delta_u$ satisfy \eqref{eq:oracle-lower}--\eqref{eq:oracle-upper} at every iteration.  If $x^\star\in\argminop_{x\in\cC}f(x)$, then
\begin{equation}
\label{eq:accelerated-bound}
f(x_N)-f(x^\star)
\le
\frac{\norm{x^\star-x_0}_2^2}{2A_N}
+(N+1)\delta_\ell+N\delta_u,
\end{equation}
and
\begin{equation}
\label{eq:A-lower}
A_N\ge\frac{N^2}{4\overline\beta}.
\end{equation}
\end{lemma}

\begin{proof}
Write $\psi_k^*:=\min_{x\in\cC}\psi_k(x)=\psi_k(z_k)$.  We first prove by induction that
\begin{equation}
\label{eq:estimate-induction}
\psi_k^*
\ge
A_k f(x_k)-E_k,
\end{equation}
where $E_0=0$ and
\begin{equation}
\label{eq:E-recursion}
E_{k+1}
:=E_k+A_k\delta_\ell+A_{k+1}\delta_u.
\end{equation}
The claim is trivial at $k=0$.  Since $\psi_k$ is $1$-strongly convex and $z_k$ minimizes it over the closed convex set $\cC$, first-order optimality and strong convexity give
\begin{equation}
\label{eq:psi-strong}
\psi_k(z_{k+1})
\ge
\psi_k(z_k)+\frac12\norm{z_{k+1}-z_k}_2^2.
\end{equation}
Let $\Delta z_k:=z_{k+1}-z_k$, $\alpha:=\alpha_{k+1}$, $A:=A_{k+1}$, and $y:=y_{k+1}$.  Combining \eqref{eq:psi-update}, \eqref{eq:psi-strong}, and the induction hypothesis yields
\begin{align*}
\psi_{k+1}^*
&\ge
A_k f(x_k)-E_k
+\frac12\norm{\Delta z_k}_2^2
+\alpha\left[f(y)+\ip{g_{k+1}}{z_{k+1}-y}\right].
\end{align*}
The lower oracle inequality at $x_k$ implies
\[
A_k f(x_k)
\ge
A_k\left[f(y)+\ip{g_{k+1}}{x_k-y}-\delta_\ell\right].
\]
Using
\begin{equation}
\label{eq:coupling-identity}
A_k(x_k-y)+\alpha(z_{k+1}-y)=\alpha\Delta z_k,
\end{equation}
we obtain
\begin{align}
\psi_{k+1}^*
&\ge
A f(y)+\alpha\ip{g_{k+1}}{\Delta z_k}
+\frac12\norm{\Delta z_k}_2^2
-E_k-A_k\delta_\ell.
\label{eq:estimate-mid}
\end{align}
On the other hand, \eqref{eq:x-update} and \eqref{eq:upper-y} give
\[
x_{k+1}-y=\frac\alpha A\Delta z_k.
\]
Applying the upper oracle inequality and using $A=\overline\beta\alpha^2$ gives
\begin{align*}
A f(x_{k+1})
&\le
A f(y)+\alpha\ip{g_{k+1}}{\Delta z_k}
+\frac{\overline\beta\alpha^2}{2A}\norm{\Delta z_k}_2^2
+A\delta_u\\
&=
A f(y)+\alpha\ip{g_{k+1}}{\Delta z_k}
+\frac12\norm{\Delta z_k}_2^2
+A\delta_u.
\end{align*}
Comparison with \eqref{eq:estimate-mid} proves \eqref{eq:estimate-induction} at $k+1$ with \eqref{eq:E-recursion}.

Evaluate $\psi_N$ at $x^\star$.  The lower oracle inequality gives
\[
f(y_i)+\ip{g_i}{x^\star-y_i}
\le f(x^\star)+\delta_\ell,
\]
so
\begin{equation}
\label{eq:psi-star-upper}
\psi_N^*
\le
\psi_N(x^\star)
\le
\frac12\norm{x^\star-x_0}_2^2
+A_Nf(x^\star)+A_N\delta_\ell.
\end{equation}
Combining \eqref{eq:estimate-induction} and \eqref{eq:psi-star-upper}, and observing that $A_k\le A_N$, gives
\begin{align*}
f(x_N)-f(x^\star)
&\le
\frac{\norm{x^\star-x_0}_2^2}{2A_N}
+\delta_\ell+\frac{E_N}{A_N}\\
&\le
\frac{\norm{x^\star-x_0}_2^2}{2A_N}
+(N+1)\delta_\ell+N\delta_u,
\end{align*}
which is \eqref{eq:accelerated-bound}.

Finally, put $s_k:=\sqrt{\overline\beta A_k}$.  Equation \eqref{eq:alpha} implies
\[
s_{k+1}^2-s_k^2=s_{k+1},
\qquad
s_{k+1}=\frac{1+\sqrt{1+4s_k^2}}{2}\ge s_k+\frac12.
\]
Since $s_0=0$, induction gives $s_N\ge N/2$, which is equivalent to \eqref{eq:A-lower}.
\end{proof}

\subsection{Choice of finite-difference accuracy}

\begin{proof}[Proof of \cref{thm:upper}]
Fix $0<\epsilon\le\beta R^2$ and set
\begin{equation}
\label{eq:upper-choices}
N:=\left\lceil2\sqrt{\frac{\beta R^2}{\epsilon}}\right\rceil,
\qquad
\delta_g:=\frac{\epsilon}{8R(N+1)},
\qquad
h:=\frac{2\delta_g}{\beta\sqrt d}.
\end{equation}
Since $\epsilon\le\beta R^2$ and $N\ge2$, one has $h\le R/2$.  At each $y_{k+1}$, use the $d+1$ exact values in \eqref{eq:fd} to form $g_{k+1}=g_h(y_{k+1})$.  By \cref{lem:fd},
\[
\norm{g_{k+1}-\nabla f(y_{k+1})}_2\le\delta_g.
\]
All base points lie in $\cC$ by construction, and all displaced points lie in $X$ by \cref{lem:fd}; the final output $x_N$ lies in $\cC$.

Apply \cref{lem:accelerated-error} with the errors in \eqref{eq:oracle-errors}.  Since $x_0=0$ and $x^\star\in B_2^d(R/2)$,
\[
\frac{\norm{x^\star-x_0}_2^2}{2A_N}
\le
\frac{\beta R^2}{N^2}
\le\frac\epsilon4.
\]
Moreover,
\[
(N+1)\delta_\ell
=(N+1)R\delta_g
=\frac\epsilon8,
\]
and
\[
N\delta_u
=
\frac{N\epsilon^2}{128\beta R^2(N+1)^2}
\le\frac\epsilon{128},
\]
where the final inequality uses $\epsilon\le\beta R^2$.  Hence
\[
f(x_N)-f^\star
\le
\left(\frac14+\frac18+\frac1{128}\right)\epsilon
<\epsilon.
\]
Each iteration uses $d+1$ calls.  Since $\beta R^2/\epsilon\ge1$,
\[
(d+1)N
\le2d\left(2\sqrt{\frac{\beta R^2}{\epsilon}}+1\right)
\le6d\sqrt{\frac{\beta R^2}{\epsilon}}.
\]
This proves \eqref{eq:upper-rate}.
\end{proof}

\section{Architecture of the lower bound}
\label{sec:overview}

The remainder of the proof establishes \cref{thm:integer,thm:main}.  The construction has a scalar chain length $m$ and a robustness radius $\rho$.  A biased maximum
\[
g_{m,\rho}(y)=\max_{1\le i\le m}\{y_i-8\rho i\}
\]
is smoothed by a Moreau envelope $h_{m,\rho}$.  Its dual problem is a concave quadratic maximization over the simplex.  When all coordinates from index $i$ onward lie in $[-\rho,\rho]$, moving any dual mass from a later coordinate to coordinate $i$ strictly improves the dual objective.  Hence every dual optimizer is supported on the first $i$ coordinates, and the smoothed value is exactly independent of the remaining tail.

The $m$ chain coordinates are embedded through an orthonormal frame $U=(u_1,\ldots,u_m)$.  Queries are grouped into blocks of size
\[
b=\left\lfloor\frac d4\right\rfloor.
\]
All replies in block $j$ are generated before $u_j$ is selected.  We then choose $u_j$ orthogonal to every point of the current block and with projection at most $\rho$ on every earlier query.  The current coordinate is therefore exactly zero, while every future coordinate is eventually made small.  Exact shielding makes the online reply independent of all unselected directions.

After the final query, one additional sentinel direction is selected orthogonal to the terminal output.  The remaining directions are chosen with small projection on all queries and the output.  Consequently the output has not reached the final chain coordinate.  A comparison point with all chain coordinates equal to $-R_0/\sqrt m$ then yields objective gap $\Omega(\beta R^2/m^2)$.

The geometric cost is the simultaneous small-projection requirement.  Since all queries have norm at most $R$, a spherical-cap union bound finds the required direction whenever
\[
d\rho^2/R^2\gtrsim\log(d^2).
\]
The optimization calibration uses $\rho\asymp R/m^{3/2}$, giving the condition $m^3\log(ed)\lesssim d$.

\section{A smooth robust value chain}
\label{sec:chain}

Fix $m\ge1$ and $\rho>0$.  Define
\begin{equation}
\label{eq:gdef}
g_{m,\rho}(y):=\max_{1\le i\le m}\{y_i-8\rho i\},
\qquad y\in\R^m,
\end{equation}
and
\begin{equation}
\label{eq:hdef}
h_{m,\rho}(y)
:=
\inf_{z\in\R^m}
\left\{
 g_{m,\rho}(z)+\frac1\rho\norm{z-y}_2^2
\right\}.
\end{equation}
This is the Moreau envelope with parameter $\rho/2$.

\subsection{Smoothness and dual representation}

\begin{lemma}[Basic Moreau properties]
\label{lem:moreau}
The function $g_{m,\rho}$ is convex and $1$-Lipschitz.  The function $h_{m,\rho}$ is convex and continuously differentiable, and
\begin{align}
0&\le g_{m,\rho}(y)-h_{m,\rho}(y)\le\frac\rho4,
\label{eq:approx}\\
\norm{\nabla h_{m,\rho}(y)}_2&\le1,
\label{eq:hgrad}\\
\Lip(\nabla h_{m,\rho})&\le\frac2\rho.
\label{eq:hsmooth}
\end{align}
If $z_y$ is the unique minimizer in \eqref{eq:hdef}, then
\begin{equation}
\label{eq:proxradius}
\norm{z_y-y}_2\le\frac\rho2.
\end{equation}
\end{lemma}

\begin{proof}
The function $g_{m,\rho}$ is the maximum of affine functions with slopes $e_1,\ldots,e_m$, hence it is convex and $1$-Lipschitz.  Choosing $z=y$ in \eqref{eq:hdef} gives $h(y)\le g(y)$.  Conversely, with $r=\norm{z-y}_2$,
\[
g(z)+\frac{r^2}{\rho}
\ge g(y)-r+\frac{r^2}{\rho}
\ge g(y)-\frac\rho4,
\]
because the minimum of $-r+r^2/\rho$ over $r\ge0$ is $-\rho/4$.  This proves \eqref{eq:approx}.

The objective in \eqref{eq:hdef} is strongly convex in $z$, so $z_y$ is unique.  First-order optimality gives an $s_y\in\partial g(z_y)$ such that
\[
s_y+\frac2\rho(z_y-y)=0.
\]
Every subgradient of a $1$-Lipschitz convex function has norm at most one, which proves \eqref{eq:proxradius}.  Standard Moreau-envelope calculus gives
\[
\nabla h(y)=\frac2\rho(y-z_y).
\]
Thus \eqref{eq:hgrad} follows, and firm nonexpansiveness of the proximal map gives \eqref{eq:hsmooth}; see, e.g., Bauschke and Combettes~\cite{bauschke2017convex}.
\end{proof}

Let
\[
\Delta_m:=\left\{p\in\R_+^m:\sum_{i=1}^m p_i=1\right\}.
\]

\begin{lemma}[Simplex dual]
\label{lem:dual}
For every $y\in\R^m$,
\begin{equation}
\label{eq:dual}
h_{m,\rho}(y)
=
\max_{p\in\Delta_m}
\left\{
\sum_{i=1}^m p_i(y_i-8\rho i)
-\frac{\rho}{4}\norm{p}_2^2
\right\}.
\end{equation}
\end{lemma}

\begin{proof}
Because the maximum of finitely many scalars equals the maximum over their convex combinations,
\[
g_{m,\rho}(z)
=
\max_{p\in\Delta_m}
\sum_{i=1}^m p_i(z_i-8\rho i).
\]
The simplex is compact and convex; the displayed expression plus $\rho^{-1}\norm{z-y}^2$ is affine in $p$ and strongly convex in $z$.  Minimax interchange is therefore valid.  For fixed $p$, the minimizer over $z$ is
\[
z=y-\frac\rho2p,
\]
and the minimum equals
\[
\sum_i p_i(y_i-8\rho i)-\frac{\rho}{4}\norm{p}_2^2.
\]
Maximizing over $p\in\Delta_m$ proves the claim.
\end{proof}

\subsection{Exact shielding}

Define the robust progress index
\begin{equation}
\label{eq:index}
i_\rho^+(y)
:=
\min\left\{
 i\in[m]: |y_k|\le\rho\ \text{for every }k\ge i
\right\},
\end{equation}
with $i_\rho^+(y)=m+1$ if the set is empty.

\begin{lemma}[Exact prefix shielding]
\label{lem:shield}
Fix $j\in[m]$.  If
\begin{equation}
\label{eq:shield-window}
|y_k|\le\rho
\qquad\text{for every }k\ge j,
\end{equation}
then every maximizer of the dual problem \eqref{eq:dual} is supported on
$\{1,\ldots,j\}$, and
\begin{equation}
\label{eq:shield}
h_{m,\rho}(y)
=
h_{m,\rho}(y_1,\ldots,y_j,0,\ldots,0).
\end{equation}
In particular, if $y_j=0$, then
\begin{equation}
\label{eq:shield-zero}
h_{m,\rho}(y)
=
h_{m,\rho}(y_1,\ldots,y_{j-1},0,\ldots,0).
\end{equation}
Consequently, if $i=i_\rho^+(y)\le m$, then the value is exactly
independent of $y_{i+1},\ldots,y_m$.
\end{lemma}

\begin{proof}
Set $c_k=y_k-8\rho k$.  For every $k>j$, condition
\eqref{eq:shield-window} gives
\[
c_j-c_k
\ge -\rho-8\rho j-(\rho-8\rho k)
=8\rho(k-j)-2\rho
\ge6\rho.
\]
Let $p$ maximize the dual problem \eqref{eq:dual}.  Suppose that
$p_k>0$ for some $k>j$.  Form $\widetilde p\in\Delta_m$ by moving all
mass $p_k$ from coordinate $k$ to coordinate $j$:
\[
\widetilde p_j=p_j+p_k,
\qquad
\widetilde p_k=0,
\qquad
\widetilde p_\ell=p_\ell\quad(\ell\notin\{j,k\}).
\]
The change in the dual objective is
\begin{align*}
&p_k(c_j-c_k)
-\frac\rho4\big[(p_j+p_k)^2-p_j^2-p_k^2\big]\\
&\qquad
=p_k(c_j-c_k)-\frac\rho2p_jp_k
\ge p_k\left(6\rho-\frac\rho2\right)>0,
\end{align*}
a contradiction.  Thus every dual maximizer is supported on
$\{1,\ldots,j\}$.

After replacing $y_{j+1},\ldots,y_m$ by zero, the same argument again
forces every dual maximizer to be supported on $\{1,\ldots,j\}$.  On
that face of the simplex, the two dual objectives coincide, which proves
\eqref{eq:shield}.  If $y_j=0$, the right-hand side of
\eqref{eq:shield} is exactly the vector displayed in
\eqref{eq:shield-zero}.  The final assertion follows by taking
$j=i_\rho^+(y)$.
\end{proof}

\begin{remark}
The proof is global and value-level.  It does not merely show that future affine pieces are inactive near one proximal point; it identifies the support of every optimizer of the exact dual problem.  This distinction is what makes the later fixed-transcript argument rigorous.
\end{remark}

\subsection{A progress-limited gap}

Set
\begin{equation}
\label{eq:R0rho}
R_0:=\frac R4,
\qquad
\rho:=\frac{R_0}{64m^{3/2}}
=\frac{R}{256m^{3/2}},
\end{equation}
and define
\begin{equation}
\label{eq:ycirc}
y^\circ:=-\frac{R_0}{\sqrt m}\1_m.
\end{equation}
Then $\norm{y^\circ}_2=R_0$.

\begin{lemma}[Gap before the final coordinate]
\label{lem:progressgap}
If $i_\rho^+(y)\le m$, then
\begin{equation}
\label{eq:progressgap}
h_{m,\rho}(y)-h_{m,\rho}(y^\circ)
\ge\frac{3R_0}{4\sqrt m}.
\end{equation}
\end{lemma}

\begin{proof}
The increasing bias implies
\[
g_{m,\rho}(y^\circ)
=-\frac{R_0}{\sqrt m}-8\rho.
\]
The condition $i_\rho^+(y)\le m$ implies $|y_m|\le\rho$, and hence
\[
g_{m,\rho}(y)
\ge y_m-8m\rho
\ge-\rho-8m\rho.
\]
Using \eqref{eq:approx},
\begin{align*}
h(y)-h(y^\circ)
&\ge g(y)-\frac\rho4-g(y^\circ)\\
&\ge\frac{R_0}{\sqrt m}-8m\rho+\frac{27\rho}{4}\\
&\ge\frac{R_0}{\sqrt m}-8m\rho.
\end{align*}
By \eqref{eq:R0rho}, $8m\rho=R_0/(8\sqrt m)$.  Thus the left-hand side is at least $7R_0/(8\sqrt m)$, which implies \eqref{eq:progressgap}.
\end{proof}

\section{Avoiding spherical caps in a subspace}
\label{sec:cap}

The delayed rotation requires a direction that is exactly orthogonal to the current block and has small projection on all earlier points.  The following elementary lemma supplies it.

\begin{lemma}[Subspace cap avoidance]
\label{lem:cap}
Let $S\subseteq\R^d$ be a $p$-dimensional subspace and let $a_1,\ldots,a_M\in B_2^d(R)$.  If $0<\rho\le R$ and
\begin{equation}
\label{eq:capcond}
p\rho^2
\ge
64R^2\log\bigl(4(M+1)\bigr),
\end{equation}
then there exists a unit vector $u\in S$ such that
\begin{equation}
\label{eq:capconclusion}
|\ip{u}{a_i}|\le\rho,
\qquad i=1,\ldots,M.
\end{equation}
\end{lemma}

\begin{proof}
The case $M=0$ is immediate.  Let $g$ be a standard Gaussian vector in $S\cong\R^p$ and put $u=g/\norm g_2$.  For fixed $a_i$, only $P_Sa_i$ matters.  On the event $\norm g_2\ge\sqrt{p/2}$, the implication
\[
|\ip{u}{a_i}|>\rho
\quad\Longrightarrow\quad
\left|\ip{g}{\frac{P_Sa_i}{\norm{P_Sa_i}_2}}\right|
>\frac{\rho}{R}\sqrt{\frac p2}
\]
holds whenever $P_Sa_i\ne0$.  A Gaussian tail bound and the standard lower-tail estimate for a chi-square variable give
\[
\Prob\{|\ip{u}{a_i}|>\rho\}
\le
2\exp\left(-\frac{p\rho^2}{4R^2}\right)
+\exp\left(-\frac p{16}\right)
\le
3\exp\left(-\frac{p\rho^2}{16R^2}\right),
\]
where the last inequality uses $\rho\le R$.  A union bound yields
\[
\Prob\left\{\max_{1\le i\le M}|\ip{u}{a_i}|>\rho\right\}
\le
3M\exp\left(-\frac{p\rho^2}{16R^2}\right)<1
\]
under \eqref{eq:capcond}.  Therefore a desired unit vector exists.  See Vershynin~\cite{vershynin2018high} for standard Gaussian and spherical concentration estimates.
\end{proof}

\begin{corollary}[A convenient parameter condition]
\label{cor:capparameter}
There is a universal $c_0>0$ such that the following holds.  Suppose $d\ge8$, $m\le d/4$, $M\le d^2/16+1$, $p\ge d/2$, and $\rho$ is given by \eqref{eq:R0rho}.  If
\[
m^3\log(ed)\le c_0d,
\]
then \eqref{eq:capcond} holds.  One may take, for example, $c_0=2^{-24}$ after increasing the harmless dimension threshold.
\end{corollary}

\begin{proof}
Equation \eqref{eq:R0rho} gives
\[
\frac{\rho^2}{R^2}=\frac{1}{2^{16}m^3},
\qquad
\frac{p\rho^2}{R^2}\ge\frac{d}{2^{17}m^3}.
\]
For $d\ge8$ and $M\le d^2/16+1$,
\[
\log(4(M+1))\le2\log(ed).
\]
The conclusion follows from \eqref{eq:capcond} with the displayed conservative value of $c_0$.
\end{proof}

\section{The batched delayed-rotation compiler}
\label{sec:compiler}

Fix a deterministic adaptive algorithm and a total budget $T$.  Let
\begin{equation}
\label{eq:blocksize}
b:=\left\lfloor\frac d4\right\rfloor,
\qquad
T=qb+s,
\qquad
0\le s<b.
\end{equation}
Assume
\begin{equation}
\label{eq:Tless}
T<bm,
\end{equation}
so $0\le q\le m-1$.  We construct orthonormal directions
\[
u_1,\ldots,u_m\in\R^d,
\qquad
U=(u_1,\ldots,u_m)\in\R^{d\times m}.
\]
Define
\begin{equation}
\label{eq:Aeta}
A:=\frac{\beta\rho}{8},
\qquad
\eta:=\frac{\beta}{4096m^2},
\end{equation}
and the direction-independent term
\begin{equation}
\label{eq:q0}
q_0(x)
:=
\frac\beta8\dist\bigl(x,B_2^d(R_0)\bigr)^2
+\frac\eta2\norm{x}_2^2.
\end{equation}

\subsection{Complete blocks}

Suppose that $u_1,\ldots,u_{j-1}$ have been selected before complete block $j$.  Every query $x$ in this block receives the response
\begin{equation}
\label{eq:fullresponse}
r(x)
:=
A h_{m,\rho}\bigl(
\ip{u_1}{x},\ldots,\ip{u_{j-1}}{x},0,\ldots,0
\bigr)+q_0(x).
\end{equation}
The algorithm may adapt within the block, but \eqref{eq:fullresponse} does not depend on $u_j$.  Consequently all points
\[
x_{j,1},\ldots,x_{j,b}
\]
are determined before $u_j$ is selected.

Set
\begin{equation}
\label{eq:Sj}
S_j
:=
\Span\{u_1,\ldots,u_{j-1},x_{j,1},\ldots,x_{j,b}\}^{\perp}.
\end{equation}
Since $j\le q\le m-1$ and $m\le d/4$,
\begin{equation}
\label{eq:Sjdim}
\dim S_j
\ge d-(j-1)-b
\ge\frac d2.
\end{equation}
Apply \cref{lem:cap} inside $S_j$ to all queries from earlier blocks and select a unit $u_j\in S_j$ such that
\begin{align}
|\ip{u_j}{x_t}|&\le\rho
&&\text{for every earlier query }x_t,
\label{eq:oldsmall}\\
\ip{u_j}{x_{j,\ell}}&=0
&&\text{for }\ell=1,\ldots,b.
\label{eq:currentzero}
\end{align}
The second relation is exact because $u_j\in S_j$.

\subsection{The incomplete block, the output, and completion}

After the $q$ complete blocks, answer the remaining $s$ queries by
\begin{equation}
\label{eq:lastresponse}
r(x)
:=
A h_{m,\rho}\bigl(
\ip{u_1}{x},\ldots,\ip{u_q}{x},0,\ldots,0
\bigr)+q_0(x).
\end{equation}
Once all responses are returned, the deterministic terminal output $\widehat x$ is fixed.  Select $u_{q+1}$ in
\[
\Span\{u_1,\ldots,u_q,x_{qb+1},\ldots,x_T,\widehat x\}^{\perp}
\]
so that it also has projection at most $\rho$ on every query from the complete blocks.  This is possible by \cref{lem:cap}: the displayed orthogonal complement has dimension at least $d/2$, because $q\le m-1$, $s+1\le b$, and $m,b\le d/4$.  Thus
\begin{align}
\ip{u_{q+1}}{x_t}&=0,
&&t=qb+1,\ldots,T,
\label{eq:sentquery}\\
\ip{u_{q+1}}{\widehat x}&=0,
\label{eq:sentout}\\
|\ip{u_{q+1}}{x_t}|&\le\rho,
&&t\le qb.
\label{eq:sentold}
\end{align}

Finally, for $k=q+2,\ldots,m$, select $u_k$ successively in
\[
\Span\{u_1,\ldots,u_{k-1}\}^{\perp}
\]
so that
\begin{equation}
\label{eq:futuresmall}
|\ip{u_k}{x_t}|\le\rho
\quad(t=1,\ldots,T),
\qquad
|\ip{u_k}{\widehat x}|\le\rho.
\end{equation}
The available subspace has dimension at least $d-m+1\ge3d/4$.  Every cap-avoidance step controls at most $T+1\le d^2/16+1$ vectors under \eqref{eq:Tless} and $m\le d/4$.  Hence \cref{cor:capparameter} validates all selections.

\begin{remark}[Fixed horizon and early stopping]
A pathwise algorithm that stops in fewer than $T$ calls may be padded by ignored repeated queries.  Therefore it is enough to prove the lower bound for fixed horizons.  The construction above never changes a response after it has been returned.
\end{remark}

\section{One fixed objective realizes the transcript}
\label{sec:fixed}

Once the frame $U$ is complete, define
\begin{equation}
\label{eq:FU}
F_U(x)
:=
A h_{m,\rho}(U^\top x)
+
\frac\beta8\dist\bigl(x,B_2^d(R_0)\bigr)^2
+
\frac\eta2\norm{x}_2^2.
\end{equation}

\begin{proposition}[Exact fixed-objective consistency]
\label{prop:transcript}
Every response generated in \cref{sec:compiler} equals the corresponding value of the single function \eqref{eq:FU}:
\[
r_t=F_U(x_t),
\qquad t=1,\ldots,T.
\]
Consequently, the query sequence and terminal output generated by the online construction are exactly the query sequence and output produced when the algorithm interacts with the fixed objective $F_U$.
\end{proposition}

\begin{proof}
Consider $x=x_{j,\ell}$ in complete block $j$ and write $y=U^\top x$.  By \eqref{eq:currentzero}, $y_j=0$.  Every direction $u_k$ with $k>j$ is selected after this query and is required, either during a later complete block or during final completion, to satisfy $|y_k|\le\rho$.  Thus \eqref{eq:shield-window} holds at the fixed prefix $j$, and \eqref{eq:shield-zero} gives
\[
h_{m,\rho}(U^\top x)
=
h_{m,\rho}\bigl(
\ip{u_1}{x},\ldots,\ip{u_{j-1}}{x},0,\ldots,0
\bigr),
\]
which is precisely the hard-core term in \eqref{eq:fullresponse}.

For a query in the incomplete block, \eqref{eq:sentquery} gives a zero at coordinate $q+1$, while \eqref{eq:futuresmall} controls every later coordinate.  Applying \eqref{eq:shield-zero} with $j=q+1$ turns \eqref{eq:FU} into \eqref{eq:lastresponse}.

Finally, use induction over time.  If all earlier replies agree with $F_U$, determinism gives the same current query; the preceding argument gives the same current reply.  Hence the complete transcript and the terminal output agree with the true interaction with $F_U$.
\end{proof}

The terminal completion also gives
\begin{equation}
\label{eq:outputprogress}
i_\rho^+(U^\top\widehat x)
\le q+1\le m,
\end{equation}
because coordinate $q+1$ is zero by \eqref{eq:sentout} and every later coordinate has magnitude at most $\rho$ by \eqref{eq:futuresmall}.

\section{Admissibility of the hard objective}
\label{sec:admissible}

\begin{proposition}[Convexity, smoothness, and strong-convexity modulus]
\label{prop:smooth}
The function $F_U$ is convex and continuously differentiable.  Moreover,
\[
\Lip(\nabla F_U)<\beta,
\qquad
\mu(F_U)=\eta=\frac{\beta}{4096m^2},
\]
where $\mu(F_U)$ denotes the largest strong-convexity modulus of $F_U$.
In particular, $F_U$ has a unique global minimizer.
\end{proposition}

\begin{proof}
Each term in \eqref{eq:FU} is convex and continuously differentiable.  Since $U^\top U=I_m$, \eqref{eq:hsmooth} gives
\[
\Lip\bigl(\nabla[A h_{m,\rho}(U^\top\cdot)]\bigr)
\le A\frac2\rho
=\frac\beta4.
\]
For a nonempty closed convex set $C$, the function $x\mapsto\frac12\dist(x,C)^2$ has gradient $x-\Pi_Cx$, which is $1$-Lipschitz.  Therefore the second term in \eqref{eq:FU} has gradient-Lipschitz constant $\beta/4$.  The final quadratic has gradient-Lipschitz constant $\eta$.  Thus
\[
\Lip(\nabla F_U)
\le\frac\beta4+\frac\beta4+\eta<\beta.
\]
The quadratic term makes $F_U$ at least $\eta$-strongly convex and coercive, so a unique minimizer exists.

It remains to show that the modulus is not larger.  Since $m<d$, choose a unit vector $v\in\operatorname{range}(U)^\perp$.  For every $|t|\le R_0$,
\[
U^\top(tv)=0,
\qquad
\dist(tv,B_2^d(R_0))=0,
\]
and hence
\[
F_U(tv)=A h_{m,\rho}(0)+\frac\eta2t^2.
\]
If $F_U$ were $\mu$-strongly convex for some $\mu>\eta$, the midpoint form of strong convexity applied to $tv$ and $-tv$ would imply
\[
F_U(tv)+F_U(-tv)-2F_U(0)\ge\mu t^2.
\]
The left-hand side equals $\eta t^2$, a contradiction for $t\ne0$.  Therefore $\mu(F_U)=\eta$.
\end{proof}

\begin{proposition}[The minimizer lies in the interior promise]
\label{prop:minloc}
Let $x_U^\star$ be the unique minimizer of $F_U$.  Then
\[
\norm{x_U^\star}_2<\frac R2.
\]
\end{proposition}

\begin{proof}
The hard-core gradient has norm at most $A$ by \eqref{eq:hgrad} and the isometry of $U$.  Write $r=\norm{x_U^\star}_2$.  If $r\le R_0$, there is nothing to prove.  Otherwise let $\nu=x_U^\star/r$.  Outside $B_2^d(R_0)$, the gradient of the distance term is
\[
\frac\beta4(r-R_0)\nu.
\]
Taking the inner product of the first-order condition $\nabla F_U(x_U^\star)=0$ with $\nu$ yields
\[
0
\ge
-A+\frac\beta4(r-R_0)+\eta r.
\]
Therefore
\[
r\le R_0+\frac{4A}{\beta}=R_0+\frac\rho2.
\]
Since $R_0=R/4$ and $\rho\le R_0/64$,
\[
r\le R_0+\frac{R_0}{128}<\frac R2.
\]
\end{proof}

\begin{remark}
The distance-squared term is used only to keep the global minimizer in the promised interior ball.  The final small quadratic supplies uniqueness and strong convexity.  Neither term depends on the hidden frame.
\end{remark}

\section{Terminal error and the integer theorem}
\label{sec:gap}

Define the comparison point
\begin{equation}
\label{eq:xcirc}
x^\circ:=Uy^\circ.
\end{equation}
Because $U$ is an isometry on $\R^m$,
\[
\norm{x^\circ}_2=\norm{y^\circ}_2=R_0,
\qquad
U^\top x^\circ=y^\circ.
\]
Thus the distance penalty vanishes at $x^\circ$, and $x^\circ\in X$.

\begin{proposition}[Output gap]
\label{prop:outputgap}
The fixed objective and terminal point constructed above satisfy
\begin{equation}
\label{eq:outputgap}
F_U(\widehat x)-F_U^\star
\ge
\frac{11}{2^{17}}\frac{\beta R^2}{m^2}
\ge
2^{-17}\frac{\beta R^2}{m^2}.
\end{equation}
\end{proposition}

\begin{proof}
By \eqref{eq:outputprogress} and \cref{lem:progressgap},
\[
h_{m,\rho}(U^\top\widehat x)-h_{m,\rho}(y^\circ)
\ge\frac{3R_0}{4\sqrt m}.
\]
Since $F_U^\star\le F_U(x^\circ)$, and the radial and quadratic terms at $\widehat x$ are nonnegative,
\begin{align*}
F_U(\widehat x)-F_U^\star
&\ge F_U(\widehat x)-F_U(x^\circ)\\
&\ge
A\frac{3R_0}{4\sqrt m}-\frac\eta2R_0^2.
\end{align*}
Using \eqref{eq:R0rho} and \eqref{eq:Aeta},
\[
A\frac{3R_0}{4\sqrt m}
=
\frac{3\beta R_0^2}{2048m^2},
\qquad
\frac\eta2R_0^2
=
\frac{\beta R_0^2}{8192m^2}.
\]
Their difference is
\[
\frac{11\beta R_0^2}{8192m^2}
=
\frac{11}{2^{17}}\frac{\beta R^2}{m^2},
\]
which proves the proposition.
\end{proof}

\begin{proof}[Proof of \cref{thm:integer}]
For $d\ge8$,
\[
b=\left\lfloor\frac d4\right\rfloor\ge\frac d8.
\]
Thus $T<dm/8$ implies $T<bm$, so the construction in \cref{sec:compiler} applies.  The condition $m^3\log(ed)\le c_0d$ validates every cap-avoidance step by \cref{cor:capparameter}.  Exact fixed-objective consistency follows from \cref{prop:transcript}; membership in $\cF_\beta(d,R)$ follows from \cref{prop:smooth,prop:minloc}; and the terminal error is \eqref{eq:outputgap}.  The exact strong-convexity modulus is the value of $\eta$ in \eqref{eq:Aeta}, by \cref{prop:smooth}.
\end{proof}

\section{From the integer theorem to the dimension--accuracy rate}
\label{sec:rate}

\begin{proof}[Proof of \cref{thm:main}]
Set
\begin{equation}
\label{eq:M}
M_\epsilon
:=
\min\left\{
\sqrt{\frac{\beta R^2}{\epsilon}},
\left(\frac{d}{\log(ed)}\right)^{1/3}
\right\}.
\end{equation}
Let $c_0$ be the constant from \cref{thm:integer}, and choose a universal $a>0$ such that
\begin{equation}
\label{eq:achoice}
a\le\frac14,
\qquad
a^3\le c_0,
\qquad
a^2\le 2^{-18}.
\end{equation}
Choose $d_0$ sufficiently large that
\[
\left(\frac{d}{\log(ed)}\right)^{1/3}\ge\frac2a
\quad\text{and}\quad
a\left(\frac{d}{\log(ed)}\right)^{1/3}\le\frac d4
\]
for all $d\ge d_0$.  Finally set $c_\epsilon=a^2/4$.

If $0<\epsilon\le c_\epsilon\beta R^2$, then
\[
\sqrt{\frac{\beta R^2}{\epsilon}}\ge\frac2a.
\]
The choice of $d_0$ therefore gives $M_\epsilon\ge2/a$.  Define
\[
m:=\lfloor aM_\epsilon\rfloor.
\]
Then
\begin{equation}
\label{eq:mcompare}
\frac a2M_\epsilon\le m\le aM_\epsilon.
\end{equation}
Moreover,
\[
m\le\frac d4,
\qquad
m^3\log(ed)\le a^3M_\epsilon^3\log(ed)\le c_0d.
\]
Thus \cref{thm:integer} applies.  Since $M_\epsilon\le\sqrt{\beta R^2/\epsilon}$ and $m\le aM_\epsilon$,
\[
2^{-17}\frac{\beta R^2}{m^2}
\ge
\frac{2^{-17}}{a^2}\epsilon
\ge2\epsilon.
\]
Hence every budget $T<dm/8$ fails with strict error greater than $\epsilon$.  By \eqref{eq:mcompare},
\[
N_{\epsilon,\mathrm{det-ad}}^{\mathrm{val}}(d,R,\beta)
\ge\frac{dm}{8}
\ge\frac a{16}dM_\epsilon.
\]
Taking $c=a/16$ and substituting \eqref{eq:M} proves \eqref{eq:main-rate}.
\end{proof}

\begin{proof}[Proof of \cref{cor:regimes}]
The first term in \eqref{eq:M} is no larger than the second exactly when
\[
Q\le\left(\frac{d}{\log(ed)}\right)^{2/3}.
\]
Substitution in \cref{thm:main} gives the two cases.
\end{proof}

\begin{proof}[Proof of \cref{cor:matching}]
Condition \eqref{eq:matching-regime} is equivalent to
\[
\frac1{c_\epsilon}
\le
\frac{\beta R^2}{\epsilon}
\le
\left(\frac{d}{\log(ed)}\right)^{2/3}.
\]
The first branch of \cref{thm:main} therefore gives the lower inequality in \eqref{eq:matching-rate}.  The upper inequality follows from \cref{thm:upper}; after enlarging the universal constant, its statement applies uniformly over the same range.
\end{proof}

\section{Discussion and limitations}
\label{sec:discussion}

\subsection{What is near-optimal}

The upper bound in \cref{thm:upper} and the first branch of \cref{thm:main} coincide up to universal constants.  Therefore the exact deterministic minimax complexity is
\[
\Theta\left(d\sqrt{\frac{\beta R^2}{\epsilon}}\right)
\]
throughout \eqref{eq:matching-regime}.  This is the precise near-optimal statement established by the paper.  It is not a claim about every accuracy scale.

\subsection{The role of batching}

A single exact real value can, in principle, carry information about every coordinate.  The proof therefore does not decompose the objective into independent coordinate tasks.  Instead, the transcript itself determines which directions are exposed.  During a block, the algorithm may choose arbitrary adaptive vectors, but all of them are generated before the next hidden direction exists.  That direction is then selected from a large orthogonal complement.  This chronological order turns $\Theta(d)$ scalar calls into at most one unit of chain progress.

\subsection{Why the transcript is genuinely fixed-objective}

A resisting-oracle sketch is insufficient if different rounds are answered by mutually incompatible functions.  Here the directions are completed after the transcript, but \cref{prop:transcript} proves that every earlier reply is exactly the value of the final $F_U$.  Since the algorithm is deterministic, equality of replies inductively forces equality of all query points and of the terminal output.  The hard objective may depend on the algorithm, as is standard in deterministic minimax lower bounds, but it is one fixed function along the realized run.

\subsection{The cubic saturation and the high-accuracy gap}

The construction uses
\[
\rho=\Theta\left(\frac{R}{m^{3/2}}\right).
\]
This scale is needed because the biased chain accumulates $m\rho$ in its offsets, while the useful comparison displacement has coordinates of size $R_0/\sqrt m$.  The cap lemma requires
\[
\frac{d\rho^2}{R^2}\gtrsim\log(T+1).
\]
Since $T=O(dm)=O(d^2)$ in the admissible range, this becomes
\[
m^3\log(ed)\lesssim d.
\]
Thus the lower bound saturates at $d^{4/3}/\log^{1/3}(ed)$.  Improving constants cannot extend the matching square-root branch to $m\asymp d$; a different shielding invariant would be required.  Known value-only localization methods give an $O(d^2\log(d+1)\log(1/\delta))$ relative-error upper bound for continuous convex-body optimization \cite{protasov1996algorithms}.  After translating relative error to the present bounded-range benchmarks, this leaves a polynomial gap between the current smooth lower bound and the high-accuracy endpoint.

\subsection{Bounded versus unrestricted queries}

The norm bound $\norm{x_t}\le R$ enters the lower bound in \cref{lem:cap}.  If an algorithm may query points of unbounded norm, a hidden direction selected from a high-dimensional subspace need not have small inner product with all prior queries.  The global definition and smoothness of $F_U$ do not repair this issue.  Hence \cref{thm:main} is not an unrestricted-query theorem.

\subsection{Deterministic versus randomized algorithms}

The frame $U$ is constructed against the realized transcript of a fixed deterministic algorithm.  This adversarial construction does not itself produce one distribution that is simultaneously hard for all deterministic policies, so it cannot be converted to a randomized lower bound by simply invoking a minimax principle.  A randomized theorem would require an additional distributional or posterior-width argument.

\subsection{Regularity and strong convexity}

The objective class uses the standard optimization notion of smoothness: $f\in C^1$ with globally Lipschitz gradient.  The Moreau envelope and squared-distance penalty need not be $C^2$ on every point, so the theorem should not be quoted for a $C^2$ or $C^\infty$ hard class without an additional smoothing argument.

The hard objectives are $\beta/(4096m^2)$-strongly convex, and the proof above shows that their global strong-convexity modulus is exactly
\[
\frac{\beta}{4096m^2}.
\]
Thus the lower bound is not caused by nonunique minimizers or flat solution sets.  Since the modulus varies with $m$, however, the theorem is not a full fixed-condition-number lower bound and does not yield a $d\sqrt\kappa\log(1/\epsilon)$ complexity characterization.

\section{Conclusion}

We established an explicit bounded-query upper bound and a new deterministic adaptive lower bound for exact-value smooth convex optimization under bounded queries.  Coordinate finite differences and an error-robust accelerated projected method use
\[
O\left(d\sqrt{\frac{\beta R^2}{\epsilon}}\right)
\]
function values.  A Moreau-smoothed max chain, exact prefix shielding, and batched delayed rotations give the lower bound
\[
\Omega\left(
 d\min\left\{
 \sqrt{\frac{\beta R^2}{\epsilon}},
 \left(\frac{d}{\log(ed)}\right)^{1/3}
 \right\}
\right).
\]
The two bounds match up to universal constants in the moderate-accuracy range \eqref{eq:matching-regime}.  Closing the high-accuracy gap, removing the bounded-query assumption, and treating randomized adaptive policies remain separate problems. The research pipeline was carried out almost entirely by
\ResearchAgentSystem{}, which also conducted a Lean-based review of the
manuscript.

\appendix

\section{Auxiliary convex-analytic facts}
\label{app:convex}

For completeness, we record two facts used in the main proof.

\begin{lemma}[Projection and squared distance]
Let $C\subseteq\R^d$ be nonempty, closed, and convex.  Then
\[
D_C(x):=\frac12\dist(x,C)^2
\]
is convex and continuously differentiable with
\[
\nabla D_C(x)=x-\Pi_Cx.
\]
Moreover, $\nabla D_C$ is $1$-Lipschitz.
\end{lemma}

\begin{proof}
The projection $\Pi_C$ is firmly nonexpansive:
\[
\norm{\Pi_Cx-\Pi_Cy}^2
\le\ip{\Pi_Cx-\Pi_Cy}{x-y}.
\]
Expanding the norm of $(I-\Pi_C)x-(I-\Pi_C)y$ shows that $I-\Pi_C$ is also firmly nonexpansive and hence nonexpansive.  Danskin's theorem, or a direct directional-derivative calculation using uniqueness of the Euclidean projection, gives the gradient formula.  Convexity follows from monotonicity of the gradient.
\end{proof}

\begin{lemma}[Chi-square lower tail]
If $g\sim N(0,I_p)$, then
\[
\Prob\left\{\norm g_2^2\le\frac p2\right\}\le e^{-p/16}.
\]
\end{lemma}

\begin{proof}
For any $t>0$, Markov's inequality gives
\[
\Prob\left\{e^{-t\norm g^2}\ge e^{-tp/2}\right\}
\le e^{tp/2}\E e^{-t\norm g^2}
=e^{tp/2}(1+2t)^{-p/2}.
\]
Taking $t=1/2$ yields
\[
\exp\left(\frac p4-\frac p2\log2\right)\le e^{-p/16}.
\]
\end{proof}

\section{A compact pseudocode description of the compiler}
\label{app:pseudocode}

\begin{algorithm}[H]
\caption{Batched delayed-rotation fixed-objective lower-bound compiler.}
\label{alg:compiler}
\begin{algorithmic}[1]
\Require A fixed deterministic $T$-call algorithm; admissible $d,m,R,\beta$; block size $b=\lfloor d/4\rfloor$.
\Ensure A transcript, an orthonormal frame $U$, and one fixed hard objective $F_U$.
\State Set $q=\lfloor T/b\rfloor$ and $s=T-qb$.
\For{$j=1,\ldots,q$}
  \State Answer the next $b$ adaptive queries using \eqref{eq:fullresponse}.
  \State After the block is fixed, choose $u_j$ orthogonal to the block and cap-avoiding all earlier queries.
\EndFor
\State Answer the final $s$ adaptive queries using \eqref{eq:lastresponse}.
\State Compute the deterministic output $\widehat x$ from the completed transcript.
\State Choose $u_{q+1}$ orthogonal to the incomplete block and $\widehat x$, and cap-avoiding all complete-block queries.
\State Complete $u_{q+2},\ldots,u_m$ by cap avoidance on all queries and $\widehat x$.
\State Define $F_U$ by \eqref{eq:FU}.
\end{algorithmic}
\end{algorithm}

\section{Constant bookkeeping}
\label{app:constants}

The proof keeps most universal constants symbolic because only their existence matters asymptotically.  One conservative bookkeeping is as follows.  The cap lemma is stated with constant $64$.  For all sufficiently large $d$, the choice $c_0=2^{-24}$ in \cref{cor:capparameter} is valid.  The integer theorem uses the block factor $1/8$ and the error constant $2^{-17}$.  In the passage to \cref{thm:main}, any
\[
a\le\min\left\{\frac14,c_0^{1/3},2^{-9}\right\}
\]
is admissible, after increasing $d_0$.  One may then take
\[
c_\epsilon=\frac{a^2}{4},
\qquad
c=\frac a{16}.
\]
These values are deliberately conservative and are not optimized.

\end{document}